\documentclass[a4paper,11pt,reqno]{amsart}
\usepackage{amsmath}
\usepackage{amstext}
\usepackage{amsbsy}
\usepackage{amsopn}
\usepackage{upref}
\usepackage{amsthm}
\usepackage{amsfonts}
\usepackage{amssymb}
\usepackage{mathrsfs}
\usepackage{times}
\allowdisplaybreaks
 \newtheorem{theorem}{Theorem}
 \newtheorem{corollary}[theorem]{Corollary}
 
 \newtheorem{lemma}[theorem]{Lemma}
\theoremstyle{definition}
 
\theoremstyle{remark}
 
\begin{document}
\title[Hermite expansions]{Hermite expansions of some tempered distributions}
\author[H.~Chihara et al]{Hiroyuki Chihara, Takashi Furuya and Takumi Koshikawa}
\address[HC]{Department of Mathematics, Faculty of Education, University of the Ryukyus, Nishihara, Okinawa 903-0213, Japan}
\email{hiroyuki.chihara@gmail.com}
\address[TF]{Graduate School of Mathematics, Nagoya University, Nagoya, Aichi 464-8602, Japan }
\email{takashi.furuya0101@gmail.com}
\address[TK]{Department of Science, Graduate School of Science and Technology for Innovation, Yamaguchi University, Yamaguchi, Yamaguchi 753-8512, Japan}
\email{tk.knku14tn050508@gmail.com}
\thanks{The first author is supported by the JSPS Grant-in-Aid for Scientific Research \#16K05221.}
\begin{abstract}
We compute Hermite expansions of some tempered distributions by using the Bargmann transform. 
In other words, we calculate the Taylor expansions of the corresponding entire functions.  
Our method of computations seems to be superior to the direct computations in the shifts of singularities and the higher dimensional cases. 
\end{abstract}
\maketitle
\section{Introduction}
\label{section:introduction}
This paper is concerned with computations of Hermite expansions of some tempered distributions. 
Let $n$ be a positive integer describing the space dimension.
We denote the set of all square integrable functions on $\mathbb{R}^n$ 
by $L^2(\mathbb{R}^n)$, which is a Hilbert space equipped with an inner product
$$
(f,g)_{L^2(\mathbb{R}^n)}
=
\int_{\mathbb{R}^n}
f(x)\overline{g(x)}
dx
$$
for $f,g \in L^2(\mathbb{R}^n)$. 
The system of Hermite functions is defined by  
$$
h_\alpha(x)
=
\frac{(-1)^{\lvert\alpha\rvert}}{\pi^{n/4}2^{\lvert{\alpha}\rvert/2}\alpha!^{1/2}}
e^{x^2/2}
\left(\frac{\partial}{\partial x}\right)^\alpha
e^{-x^2},
\quad
\alpha\in(\mathbb{N}\cup\{0\})^n, 
$$
where
$\mathbb{N}$ is the set of all positive integers, 
and 
for $\alpha=(\alpha_1,\dotsc,\alpha_n)\in(\mathbb{N}\cup\{0\})^n$, 
$x=(x_1,\dotsc,x_n)\in\mathbb{R}^n$ and $z=(z_1,\dotsc,z_n)\in\mathbb{C}^n$, 
set
$\lvert\alpha\rvert=\alpha_1+\dotsb+\alpha_n$, 
$\alpha!=\alpha_1!\dotsb\alpha_n!$, 
$zx=z_1x_1+\dotsb+z_nx_n$,
$x^2=xx$,
$\lvert{z}\rvert^2=z\bar{z}$ 
and
$$
\left(\frac{\partial}{\partial x}\right)^\alpha
=
\frac{\partial^{\alpha_1}}{\partial x_1^{\alpha_1}}
\dotsb
\frac{\partial^{\alpha_n}}{\partial x_n^{\alpha_n}}.
$$
It is well-known that the system of Hermite functions is 
a complete orthonormal system of the Hilbert space $L^2(\mathbb{R}^n)$, 
that is, for any $f \in L^2(\mathbb{R}^n)$, 
$$
f=\sum_{\alpha}(f,h_\alpha)_{L^2(\mathbb{R}^n)}h_\alpha
\quad
\text{in}
\quad
L^2(\mathbb{R}^n)
$$
holds, where the summation is taken over all multi-indices.  
\par
Let $\mathscr{S}(\mathbb{R}^n)$ be the Schwartz class on $\mathbb{R}^n$, 
and let $\mathscr{S}^\prime(\mathbb{R}^n)$ be its topological dual, 
that is, the set of all tempered distributions on $\mathbb{R}^n$. 
In the celebrated paper \cite{simon}, 
Simon studied the Hermite expansions of tempered distributions
$$
T=\sum_{\alpha}\langle{T,h_\alpha}\rangle h_\alpha
\quad
\text{in}
\quad
\mathscr{S}^\prime(\mathbb{R}^n)
$$
for any $T\in\mathscr{S}^\prime(\mathbb{R}^n)$, 
where $\langle\cdot,\cdot\rangle$ denotes the pairing of
$\mathscr{S}^\prime(\mathbb{R}^n)$ and $\mathscr{S}(\mathbb{R}^n)$.
Note that all the Hermite functions $h_\alpha$ are real-valued.
In \cite{simon} Simon characterize
the decay order of the Fourier coefficients of rapidly decreasing functions,
and the growth order of the Fourier coefficients of tempered distributions.
More precisely he proved the following.
\begin{theorem}[Simon \cite{simon}]
\label{theorem:simon}
\quad
\begin{itemize}
\item
If $\phi \in \mathscr{S}(\mathbb{R}^n)$, then 
$\displaystyle\sum_{\alpha}\lvert(\phi,h_\alpha)_{L^2(\mathbb{R}^n)}\rvert^2(1+\lvert\alpha\rvert)^m<\infty$    
for all $m =0,1,2,\dotsc$.
\item
Let 
$\{a_\alpha\}_{\alpha\in(\mathbb{N}\cup\{0\})^n}$ 
be a sequence of complex numbers. 
If 
$\displaystyle\sum_{\alpha} \lvert{a_\alpha}\rvert^2(1+\lvert\alpha\rvert)^m<\infty$ 
for all $m =0,1,2,\dotsc$, 
then a series 
$\displaystyle\sum_\alpha a_\alpha h_\alpha$ 
converges in $\mathscr{S}(\mathbb{R}^n)$. 
\item 
If $T \in \mathscr{S}^\prime(\mathbb{R}^n)$, 
then there exist $C>0$ and $m\in\mathbb{N}\cup\{0\}$ such that 
$\lvert\langle{T,h_\alpha}\rangle\rvert\leqslant C(1+\lvert\alpha\rvert)^m$ 
for all $\alpha \in (\mathbb{N}\cup\{0\})^n$. 
\item 
Let 
$\{a_\alpha\}_{\alpha\in(\mathbb{N}\cup\{0\})^n}$ 
be a sequence of complex numbers. 
If there exist $C>0$ and $m\in\mathbb{N}\cup\{0\}$ such that 
$\lvert{a_\alpha}\rvert\leqslant C(1+\lvert\alpha\rvert)^m$ 
for all $\alpha \in (\mathbb{N}\cup\{0\})^n$,       
then 
$\mathscr{S}(\mathbb{R}^n)\ni\phi \mapsto \displaystyle\sum_{\alpha}a_n(\phi,h_\alpha)_{L^2(\mathbb{R}^n)}$ 
defines a tempered distribution.       
\end{itemize}
\end{theorem}
\par
There have been few concrete examples of Hermite expansions of tempered distributions 
until recently. 
Quite recently, in \cite{kagawa} Kagawa computed Hermite expansions of
some tempered distributions on $\mathbb{R}$:
the Dirac measure at the origin, the Heaviside function,
the signature function, the principal value of $1/x$ and etc,
which have a single point singularity at the origin.
His method of computation is based on the direct computation of the Fourier coefficients of Hermite expansions and the Fourier transform of the Hermite functions. 
\par
The purpose of the present paper is to propose an alternative method of
calculating Hermite expansions of tempered distributions on the Euclidean space. 
Our method of proof is based on the Bargmann transform of tempered distributions.
The Bargmann transform of $u \in \mathscr{S}(\mathbb{R}^n)$ is defined by
$$
Bu(z)
=
2^{-n/2}
\pi^{-3n/4}
\int_{\mathbb{R}^n}
e^{-(z^2/4-zx+x^2/2)}
u(x)
dx,
\quad
z\in\mathbb{C}^n. 
$$
The Bargmann transform can be defined for tempered distributions 
since its integral kernel is a Schwartz function for any fixed $z\in\mathbb{C}^n$.
A Bargmann transformation of a Hermite expansion becomes 
a Taylor expansion of an entire function on $\mathbb{C}^n$ 
because the integral kernel of the Bargmann transform is 
the generating function of the Hermite functions. 
We believe that if we use the Bargmann transform, it becomes relatively easy to deal with a single point singularity which is not necessarily located at the origin, and higher dimensional cases. 
\par
Here we recall elementary properties of Bargmann transform needed later. 
Let $L^2_B(\mathbb{C}^n)$ be a Hilbert space of measurable functions on $\mathbb{C}^n$ 
equipped with an inner product
$$
(U,V)_{L^2_B(\mathbb{C}^n)}
= 
\int_{\mathbb{C}^n}
U(z)\overline{V(z)}
e^{-\lvert{z}\rvert^2/2}
L(dz)
$$
for $U,V \in L^2_B(\mathbb{C}^n)$,
where $L(dz)$ is the Lebesgue measure on $\mathbb{C}^n\simeq\mathbb{R}^{2n}$.
We denote by $H_B(\mathbb{C}^n)$ 
the set of all holomorphic functions in $L^2_B(\mathbb{C}^n)$. 
Then $H_B(\mathbb{C}^n)$ also becomes a Hilbert space since $H_B(\mathbb{C}^n)$ is a closed subspace of $L^2_B(\mathbb{C}^n)$. 
It is well-known that the Bargmann transform $B$ is a Hilbert space isometry of $L^2(\mathbb{R}^n)$ onto $H_B(\mathbb{C}^n)$, and the inverse mapping is given by 
$$
B^\ast U(x)
=
2^{-n/2}\pi^{-3n/4}
\int_{\mathbb{C}^n}
e^{-(\bar{z}^2/4-\bar{z}x-x^2/2)}
U(z)
e^{-\lvert{z}\rvert^2/2}
L(dz),
\quad
U \in H_B(\mathbb{C}^n). 
$$
Note that $B^\ast$ is extended on $L^2_B(\mathbb{C}^n)$ onto $L^2(\mathbb{R}^n)$.  
Set $\varphi_\alpha=Bh_\alpha$ for all $\alpha\in(\mathbb{N}\cup\{0\})^n$,
that is, $h_\alpha=B^\ast\varphi_\alpha$. 
The family of holomorphic functions
$\{\varphi_\alpha\}_{\alpha\in(\mathbb{N}\cup\{0\})^n}$ 
is a complete orthonormal system of $H_B(\mathbb{C}^n)$
since $B$ is a Hilbert space isometry.
Indeed $\varphi_\alpha$ is given by
$$
\varphi_\alpha(z)
=
\frac{z^\alpha}{\sqrt{(2\pi)^n 2^{\lvert\alpha\rvert} \alpha!}},
\quad
\alpha\in(\mathbb{N}\cup\{0\})^n.
$$
This implies that if the Taylor expansion of $BT(z)$ for $T\in\mathscr{S}^\prime(\mathbb{R}^n)$ is given by 
$$
BT(z)=\sum_{\alpha}b_\alpha z^\alpha
\quad\text{in}\quad
B\bigl(\mathscr{S}^\prime(\mathbb{R}^n)\bigr),
\quad
\{b_\alpha\}_{\alpha\in(\mathbb{N}\cup\{0\})^n}\subset\mathbb{C}, 
$$
then the Hermite expansion of $T$ is given by 
$$
T 
= 
\sum_{\alpha} 
\Bigl(
\pi^{n/2}\ 2^{(\lvert\alpha\rvert+n)/2}\ \alpha!^{1/2}
\cdot
b_\alpha
\Bigr) 
h_\alpha
\quad\text{in}\quad
\mathscr{S}^\prime(\mathbb{R}^n). 
$$
For more details about the Bargmann transform and related topics, see, e.g.,
\cite{folland}, 
\cite{sjoestrand}, 
\cite{wong1}, 
\cite{wong2}, 
\cite{chihara1}, 
\cite{chihara2} 
and references therein. 
We shall compute the Taylor expansions of the Bargmann transformations of some tempered distributions and give their Hermite expansions.
\par
The plan of the present paper is as follows.
Section~\ref{section:preliminaries} prepares some elementary facts used later. 
Section~\ref{section:one-dimensional-cases} computes Hermite expansions of some tempered distributions on the real axis: 
$x^p$ ($p\in\mathbb{N}\cup\{0\}$), 
$x_+^\lambda$ ($\operatorname{Re}\lambda>-1$), 
$\lvert{x}\rvert^\lambda$  ($\operatorname{Re}\lambda>-1$), 
$\operatorname{sgn}(x)$, 
the Dirac measure $\delta_c$ ($c\in\mathbb{R}$), 
$\operatorname{vp}1/x$,  
and 
$(x\pm{i}0)^\lambda$ for $\operatorname{Re}(\lambda)>-1$ or $\lambda=-1$, 
where $x_+^\lambda$ is a function of $x\in\mathbb{R}$ defined by 
$x_+=x^\lambda$ for $x>0$ and $x_\pm^\lambda=0$ for $x<0$. 
Finally Section~\ref{section:higher-dimensional-cases} computes Hermite expansions of
tempered distributions on the Euclidean space whose dimension is strictly more than one:
the Dirac measure $\delta_c$ ($c\in\mathbb{R}^n$),
$\lvert{x}\rvert^{\lambda}$ ($\operatorname{Re}(\lambda>-n)$), 
$\operatorname{vp} K(x)/\lvert{x}\rvert^n$ 
with some vanishing condition of $K$ on the unit sphere 
$\mathbb{S}^{n-1}=\{x\in\mathbb{R}^n \ \vert \ \lvert{x}\rvert=1\}$, 
and the standard volume element of spheres centered at the origin. 
\par
The present paper is based on the activity for undergraduate thesis 
of the second and the third authors 
at College of Mathematics, University of Tsukuba. 
This was supervised by the first author. 
The authors are grateful for the research environment 
at Institute of Mathematics, University of Tsukuba. 
The authors would like to thank the referee for reading our manuscript carefully. 
\section{Preliminaries}
\label{section:preliminaries}
In this section we confirm some preliminary facts by elementary computations. 
Here we recall the definition of the Gamma function
$$
\Gamma(\lambda)
=
\int_0^\infty
t^{\lambda-1}
e^{-t}
dt,
\quad
\operatorname{Re}(\lambda)>0.
$$
We first check some moments of the one-dimensional standard normal distribution. 
\begin{lemma}
\label{theorem:moments} 
For $\operatorname{Re}(\lambda)>-1$,
$$
I(\lambda)
=
\int_0^\infty 
x^{\lambda}
e^{-x^2/2}
dx
=
2^{(\lambda-1)/2} 
\Gamma\left(\frac{\lambda+1}{2}\right). 
$$
In particular, for $l=0,1,2,\dotsc$, 
$$
I(2l)
=
\frac{(2l)!}{2^{l+1/2}l!}\sqrt{\pi},
\quad
I(2l+1)
=
2^ll!.
$$
\end{lemma}
\begin{proof}
These computations are basically due to 
the change of variable $x=\sqrt{2t}$, 
and well-known facts  
$\Gamma(1/2)=\sqrt{\pi}$, $\Gamma(1)=1$ and  
$\Gamma(\lambda+1)=\lambda\Gamma(\lambda)$ for $\operatorname{Re}(\lambda)>0$.  
\end{proof}
Let $n$ be an integer not smaller than two. 
In $\mathbb{R}^n$ we need moments of the unit sphere $\mathbb{S}^{n-1}$ 
for some tempered distributions. 
We denote the volume element of hypersurfaces by $d\sigma$. 
Suppose that $K \in C^\infty(\mathbb{S}^{n-1})$. 
Set
$$
M(n,K,\alpha)
=
\int_{\mathbb{S}^{n-1}}
\omega^\alpha
K(\omega)
d\sigma(\omega),
\quad
\alpha\in(\mathbb{N}\cup\{0\})^n.
$$
Note that if $\alpha/2 \not\in (\mathbb{N}\cup\{0\})^n$ then 
$$
M(n,1,\alpha)
=
\int_{\mathbb{S}^{n-1}}
\omega^\alpha
d\sigma(\omega)
=
0.
$$
It is possible to compute $M(n,1,2\alpha)$ for $\alpha\in(\mathbb{N}\cup\{0\})^n$ 
by using the polar coordinates in $\mathbb{R}^n$. 
Unfortunately, however,
it seems to be very difficult to show the results of computations of
$M(n,1,2\alpha)$ by using the multi-index $\alpha$. 
Here we show $M(n,1,2\alpha)$ for $n=2,3$. 
\begin{lemma}
\label{theorem:sphere}
For $k,l,m=0,1,2,\dotsc$, 
\begin{align*}
  M\bigl(2,1,(2k,2l)\bigr)
& =
  2\pi
  \sum_{\mu=0}^l
  \frac{(2k+2\mu)! l!}{(k+\mu)!^2 \mu! (l-\mu)!}
  \cdot
  \frac{(-1)^\mu}{2^{2k+2\mu}},
\\
  M\bigl(3,1,(2k,2l,2m)\bigr)
& =
  4\pi
  \sum_{\nu=0}^{l+m}
  \frac{(l+m)!}{\nu!(l+m-\nu)!}
  \cdot
  \frac{(-1)^\nu}{2k+2\nu+1}
\\ 
& \quad \ \  \times
  \sum_{\mu=0}^m
  \frac{(2l+2\mu)! m!}{(l+\mu)!^2 \mu! (m-\mu)!}
  \cdot
  \frac{(-1)^\mu}{2^{2l+2\mu}}.
\end{align*} 
\end{lemma}
\begin{proof}
By using the polar coordinates, we have
\begin{align*}
  M\bigl(2,1,(2k,2l)\bigr)
& =
  \int_0^{2\pi}
  \cos^{2k}\theta
  \cdot
  \sin^{2l}\theta
  d\theta,
\\
  M\bigl(3,1,(2k,2l,2m)\bigr) 
& =
  \int_0^\pi 
  \left(
  \int_0^{2\pi}
  \cos^{2k}\theta
  \cdot
  \sin^{2l}\theta \cos^{2l}\varphi
  \cdot
  \sin^{2m}\theta \sin^{2m}\varphi
  d\varphi 
  \right)
  \sin\theta
  d\theta
\\
& =
  \int_0^\pi
  \cos^{2k}\theta \sin^{2l+2m+1}\theta
  d\theta
  \cdot
  M\bigl(2,1,(2l,2m)\bigr).
\end{align*}
Then we can compute these. 
\end{proof}
%
%
\section{One dimensional cases}
\label{section:one-dimensional-cases}
In this section we compute Hermite expansions 
of some tempered distributions on $\mathbb{R}$. 
Some results in the present section were first proved by Kagawa in \cite{kagawa}.
We give alternative proof of them. 
\par
First we consider monomials $x^p$ ($p\in\mathbb{N}\cup\{0\}$). 
The case $p=0$, that is, $x^0=1$, was computed in \cite[Lemma~2.1]{kagawa}. 
\begin{theorem}
\label{theorem:monomials1} 
Let $p$ be a nonnegative integer. 
If we set 
$$
x^p
=
\sum_{k=0}^\infty
a_k(p)h_k
\quad\text{in}\quad 
\mathscr{S}^\prime(\mathbb{R}),
$$
then for any nonnegative integers $q$ and $l$,  
\begin{align*}
  a_{2l}(2q)
& =
  \pi^{1/4}
  2^{-q+1/2}
  (2l)!^{1/2}
  \sum_{m=0}^l
  \frac{(2l-2m+2q)!}{m! (2l-2m)! (l-m+q)!}
  \cdot
  \left(-\frac{1}{2}\right)^m,
\\ 
  a_{2l+1}(2q)
& =
  0,
\\
  a_{2l}(2q+1)
& =
  0,
\\
  a_{2l+1}(2q+1)
& =
  \pi^{1/4}
  2^{-q}
  (2l+1)!^{1/2}
  \sum_{m=0}^l
  \frac{(2l-2m+2q+2)!}{m! (2l-2m+1)! (l-m+q+1)!}
  \cdot
  \left(-\frac{1}{2}\right)^m.
\end{align*}
In particular 
\begin{align*}
  1
& =
  \sum_{l=0}^\infty
  \frac{\pi^{1/4} (2l)!^{1/2}}{2^{l-1/2} l!} 
  h_{2l} 
  \quad\text{in}\quad
  \mathscr{S}^\prime(\mathbb{R}),
\\
  x
& =
  \sum_{l=0}^\infty
  \frac{\pi^{1/4} (2l+1)!^{1/2}}{2^{l-1} l!} 
  h_{2l+1}
  \quad\text{in}\quad
  \mathscr{S}^\prime(\mathbb{R}).
\end{align*}
\end{theorem}
\begin{proof}
We compute the Taylor expansion of $B(x^p)(z)$ 
for an arbitrary nonnegative integer $p$. 
By using the Taylor expansion of the exponential function, we have 
\begin{align}
  B(x^p) 
& =
  2^{-1/2} \pi^{-3/4} e^{-z^2/4} 
  \int_{\mathbb{R}}
  e^{zx-x^2/2}x^p
  dx
\nonumber
\\
& =
  2^{-1/2} \pi^{-3/4} 
  \sum_{\mu=0}^\infty
  \frac{(-1)^\mu z^{2\mu}}{\mu! 2^{2\mu}} 
  \sum_{\nu=0}^\infty
  \frac{z^\nu}{\nu!}
  \int_{\mathbb{R}}
  x^{p+\nu}
  e^{-x^2/2}
  dx.
\label{equation:4-1}
\end{align}
The integration in the last term vanishes unless $p+\nu$ is an even integer. 
We split our computations into two cases according to the parity of $p$. 
We make use of Lemma~\ref{theorem:moments} as 
$2I(2l)=2^{1/2}\pi^{1/2}(2l)!/2^ll!$. 
\par
For $p=2q$ ($q=0,1,2,\dotsc$), \eqref{equation:4-1} becomes 
\begin{align*}
  B(x^{2q})(z)
& =
  2^{-1/2} \pi^{-3/4} 
  \sum_{\mu=0}^\infty
  \frac{(-1)^\mu z^{2\mu}}{\mu! 2^{2\mu}} 
  \sum_{\nu=0}^\infty
  \frac{z^{2\nu}}{(2\nu)!}
  \int_{\mathbb{R}}
  x^{2(q+\nu)}
  e^{-x^2/2}
  dx
\\
& =
  2^{-1/2} \pi^{-3/4} 
  \sum_{\mu=0}^\infty
  \frac{(-1)^\mu z^{2\mu}}{\mu! 2^{2\mu}} 
  \sum_{\nu=0}^\infty
  \frac{2I(2p+2\nu) z^{2\nu}}{(2\nu)!}
\\
& =
  2^{-1/2} \pi^{-3/4} 
  \sum_{\mu=0}^\infty
  \frac{(-1)^\mu z^{2\mu}}{\mu! 2^{2\mu}} 
  \sum_{\nu=0}^\infty
  \frac{2^{1/2} \pi^{1/2} (2q+2\nu)! z^{2\nu}}{2^{q+\nu} (q+\nu)! (2\nu)!}
\\
& =
  \pi^{-1/4} 
  \sum_{\mu=0}^\infty
  \frac{(-1)^\mu z^{2\mu}}{\mu! 2^{2\mu}} 
  \sum_{\nu=0}^\infty
  \frac{(2q+2\nu)! z^{2\nu}}{2^{q+\nu} (q+\nu)! (2\nu)!}.
\end{align*}
If we rearrange the order of summations by setting 
$l=\mu+\nu$ and $m=\mu$, we have 
\begin{align*}
  B(x^{2q})(z)
& =
  \sum_{l=0}^\infty
  \left\{
  \pi^{-1/4}
  \sum_{m=0}^l
  \frac{(2l-2m+2q)!}{m! (l-m+q)! (2l-2m)!}
  \cdot
  \frac{(-1)^m}{2^{l+m+q}}
  \right\}
  z^{2l}
\\
& =
  \sum_{l=0}^\infty
  \left\{
  \pi^{-1/4} 2^{-l-q}
  \sum_{m=0}^l
  \frac{(2l-2m+2q)!}{m! (l-m+q)! (2l-2m)!}
  \cdot
  \left(-\frac{1}{2}\right)^m
  \right\}
  z^{2l}. 
\end{align*}
By using this, we obtain for any nonnegative integer $l$, 
$a_{2l+1}(2q)=0$ and 
\begin{align*}
  a_{2l}(2q)
& =
  \pi^{1/2} 2^{l+1/2} (2l)!^{1/2}
  \times
  \pi^{-1/4} 2^{-l-q}
  \sum_{m=0}^l
  \frac{(2l-2m+2q)!}{m! (l-m+q)! (2l-2m)!}
  \cdot
  \left(-\frac{1}{2}\right)^m
\\
& =
  \pi^{1/4} 2^{-q+1/2} (2l)!^{1/2}
  \sum_{m=0}^l
  \frac{(2l-2m+2q)!}{m! (l-m+q)! (2l-2m)!}
  \cdot
  \left(-\frac{1}{2}\right)^m.  
\end{align*}
\par
For $p=2q+1$ ($q=0,1,2,\dotsc$), \eqref{equation:4-1} becomes 
\begin{align*}
  B(x^{2q+1})(z)
& =
  2^{-1/2} \pi^{-3/4} 
  \sum_{\mu=0}^\infty
  \frac{(-1)^\mu z^{2\mu}}{\mu! 2^{2\mu}} 
  \sum_{\nu=0}^\infty
  \frac{z^{2\nu+1}}{(2\nu+1)!}
  \int_{\mathbb{R}}
  x^{2(q+\nu+1)}
  e^{-x^2/2}
  dx
\\
& =
  2^{-1/2} \pi^{-3/4} 
  \sum_{\mu=0}^\infty
  \frac{(-1)^\mu z^{2\mu}}{\mu! 2^{2\mu}} 
  \sum_{\nu=0}^\infty
  \frac{2I(2p+2\nu+2) z^{2\nu+1}}{(2\nu+1)!}
\\
& =
  2^{-1/2} \pi^{-3/4} 
  \sum_{\mu=0}^\infty
  \frac{(-1)^\mu z^{2\mu}}{\mu! 2^{2\mu}} 
  \sum_{\nu=0}^\infty
  \frac{2^{1/2} \pi^{1/2} (2q+2\nu+2)! z^{2\nu+1}}{2^{q+\nu+1} (q+\nu+1)! (2\nu+1)!}
\\
& =
  \pi^{-1/4} 
  \sum_{\mu=0}^\infty
  \frac{(-1)^\mu z^{2\mu}}{\mu! 2^{2\mu}} 
  \sum_{\nu=0}^\infty
  \frac{(2q+2\nu+2)! z^{2\nu+1}}{2^{q+\nu+1} (q+\nu+1)! (2\nu+1)!}.
\end{align*}
If we rearrange the order of summations by setting 
$l=\mu+\nu$ and $m=\mu$, we have 
\begin{align*}
&  B(x^{2q+1})(z)
\\
  =
& \sum_{l=0}^\infty
  \left\{
  \pi^{-1/4}
  \sum_{m=0}^l
  \frac{(2l-2m+2q+2)!}{m! (l-m+q+1)! (2l-2m+1)!}
  \cdot
  \frac{(-1)^m}{2^{l+m+q+1}}
  \right\}
  z^{2l+1}
\\
  =
& \sum_{l=0}^\infty
  \left\{
  \pi^{-1/4} 2^{-l-q-1}
  \sum_{m=0}^l
  \frac{(2l-2m+2q+2)!}{m! (l-m+q+1)! (2l-2m+1)!}
  \cdot
  \left(-\frac{1}{2}\right)^m
  \right\}
  z^{2l+1}. 
\end{align*}
By using this, we obtain that for any nonnegative integer $l$, 
$a_{2l}(2q+1)=0$ and 
\begin{align*}
&  a_{2l+1}(2q+1)
\\
  =
& \pi^{1/2} 2^{l+1} (2l+1)!^{1/2}
  \times
  \pi^{-1/4} 2^{-l-q-1}
\\
  \times
& \sum_{m=0}^l
  \frac{(2l-2m+2q+2)!}{m! (l-m+q+1)! (2l-2m+1)!}
  \cdot
  \left(-\frac{1}{2}\right)^m
\\
  =
& \pi^{1/4} 2^{-q} (2l+1)!^{1/2}
  \sum_{m=0}^l
  \frac{(2l-2m+2q+2)!}{m! (l-m+q+1)! (2l-2m+1)!}
  \cdot
  \left(-\frac{1}{2}\right)^m.  
\end{align*}
\par
For special cases $p=0$ and $p=1$, we have 
\begin{align*}
  a_{2l}(0)
& =
  \pi^{1/4} 2^{1/2} (2l)!^{1/2} 
  \sum_{m=0}^l
  \frac{1}{m!(l-m)!}
  \left(-\frac{1}{2}\right)^m
\\
& =
  \frac{\pi^{1/4} 2^{1/2} (2l)!^{1/2}}{l!}
  \left(1-\frac{1}{2}\right)^l
  =
  \frac{\pi^{1/4} (2l)!^{1/2}}{2^{l-1/2}l!},
\\
  a_{2l+1}(1)
& =
  \pi^{1/4} (2l+1)!^{1/2}
  \sum_{m=0}^l
  \frac{2}{m!(l-m)!}
  \left(-\frac{1}{2}\right)^m
\\
& =
  \frac{\pi^{1/4} (2l+1)!^{1/2}}{l!}
  2
  \left(1-\frac{1}{2}\right)^l
  =
  \frac{\pi^{1/4} (2l+1)!^{1/2}}{2^{l-1} l!}   
\end{align*}
for any nonnegative integer $l$. 
This completes the proof.
\end{proof}
Let $Y(x)$ be the Heaviside function defined by 
$Y(x)=1$ for $x>0$ and $Y(x)=0$ for $x<0$. 
We consider the complex power $x_+^\lambda$ for $\operatorname{Re}(\lambda)>-1$, 
which is defined by 
$x_+^\lambda=x^\lambda$ for $x>0$ and $x_+^\lambda=0$ for $x<0$. 
Note that $Y(x) \equiv x_+^0$. 
The cases of $\lambda \in \mathbb{N}\cup\{0\}$ were computed in \cite[Theorem~2.2, Corollary~2.3]{kagawa}. 
We denote by $[t]$ the largest integer not greater than $t\in\mathbb{R}$. 
\begin{theorem}
\label{theorem:complex-power}
Let $\lambda$ be a complex number satisfying $\operatorname{Re}(\lambda)>-1$. 
If we set 
$$
x_+^\lambda 
= 
\sum_{k=0}^\infty 
a_k(\lambda) 
h_k 
\quad\text{in}\quad
\mathscr{S}^\prime(\mathbb{R}), 
$$
then for any nonnegative integer $k$, 
$$
a_k(\lambda)
= 
\pi^{-1/4} 
2^{k+(\lambda-1)/2} 
k!^{1/2} 
\sum_{m=0}^{[k/2]} 
\cfrac{\Gamma\left(\dfrac{k+\lambda+1}{2}-m\right)}{m! (k-2m)!}
\cdot
\frac{(-1)^m}{2^{3m}}.
$$
In particular, regarding $x_+^0=Y(x)$, we have for any nonnegative integer $l$, 
\begin{align*}
  a_{2l}(0)
& =
  \frac{\pi^{1/4} (2l)!^{1/2}}{2^{l+1/2} l!},
\\
  a_{2l+1}(0)
& =
  \pi^{-1/4} 2^{2l+1/2} (2l+1)!^{1/2} 
  \sum_{m=0}^l
  \frac{(l-m)!}{m! (2l-2m+1)!}
  \cdot
  \frac{(-1)}{2^{3m}}.
\end{align*}
\end{theorem}
\begin{proof}
In the same way as the proof of Theorem~\ref{theorem:monomials1}, we have 
\begin{align*}
  B(x_+^\lambda)(z)
& =
  2^{-1/2} \pi^{-3/4} e^{-z^2/4} 
  \int_{\mathbb{R}}
  e^{zx-x^2/2} 
  x_+^\lambda 
  dx
\\
& =
  2^{-1/2} \pi^{-3/4} e^{-z^2/4} 
  \int_0^\infty
  e^{zx-x^2/2}
  x^\lambda 
  dx
\\
& =
  2^{-1/2} \pi^{-3/4}
  \sum_{\mu=0}^\infty
  \frac{(-1)^\mu z^{2\mu}}{\mu! 2^{2\mu}} 
  \sum_{\nu=0}^\infty 
  \frac{z^\nu}{\nu!} 
  \int_0^\infty 
  x^{\lambda+\nu}
  e^{-x^2/2}
  dx
\\
& =
  2^{-1/2} \pi^{-3/4}
  \sum_{\mu=0}^\infty
  \frac{(-1)^\mu z^{2\mu}}{\mu! 2^{2\mu}} 
  \sum_{\nu=0}^\infty 
  \frac{z^\nu}{\nu!} 
  \cdot
  2^{(\lambda+\nu-1)/2}
  \Gamma\left(\frac{\lambda+\nu+1}{2}\right)
\\
& =
  \pi^{-3/4} 2^{\lambda/2-1} 
  \sum_{\mu=0}^\infty
  \frac{(-1)^\mu z^{2\mu}}{\mu! 2^{2\mu}} 
  \sum_{\nu=0}^\infty 
  \cfrac{2^{\nu/2} \Gamma\left(\dfrac{\lambda+\nu+1}{2}\right) z^\nu}{\nu!}. 
\end{align*}
If we rearrange the order of summations by setting 
$k=2\mu+\nu$ and $m=\mu$, we have 
\begin{align*}
  B(x_+^\lambda)(z)
& =
  \sum_{k=0}^\infty
  \left\{
  \pi^{-3/4} 2^{\lambda/2-1} 
  \sum_{m=0}^{[k/2]}
  \frac{ \Gamma\left(\dfrac{k+\lambda+1}{2}-m\right)}{m! (k-2m)!}
  \cdot
  \frac{(-1)^m 2^{k/2-m}}{2^{2m}}
  \right\}
  z^k 
\\
& =
  \sum_{k=0}^\infty
  \left\{
  \pi^{-3/4} 2^{(k+\lambda)/2-1} 
  \sum_{m=0}^{[k/2]}
  \frac{ \Gamma\left(\dfrac{k+\lambda+1}{2}-m\right)}{m! (k-2m)!}
  \cdot
  \frac{(-1)^m}{2^{3m}}
  \right\}
  z^k. 
\end{align*}
By using this, we obtain that for any nonnegative integer $k$, 
\begin{align*}
  a_k(\lambda)
& =
  \pi^{1/2} 2^{k/2+1/2} k!^{1/2}
  \times
  \pi^{-3/4} 2^{(k+\lambda)/2-1} 
  \sum_{m=0}^{[k/2]}
  \frac{ \Gamma\left(\dfrac{k+\lambda+1}{2}-m\right)}{m! (k-2m)!}
  \cdot
  \frac{(-1)^m}{2^{3m}}
\\
& =
  \pi^{-1/4} 2^{k+(\lambda-1)/2} k!^{1/2} 
  \sum_{m=0}^{[k/2]}
  \frac{ \Gamma\left(\dfrac{k+\lambda+1}{2}-m\right)}{m! (k-2m)!}
  \cdot
  \frac{(-1)^m}{2^{3m}}. 
\end{align*}
\par
Recall Lemma~\ref{theorem:moments} and note that 
$\Gamma(l+1/2)=2^{-l+1/2}I(2l)=\pi^{1/2}(2l)!/2^{2l}l!$. 
For the spacial case $\lambda=0$, we have 
\begin{align*}
  a_{2l}(0)
& =
  \pi^{-1/4} 2^{2l-1/2} (2l)!^{1/2} 
  \sum_{m=0}^{l}
  \frac{ \Gamma\left(l-m+\dfrac{1}{2}\right)}{m! (2l-2m)!}
  \cdot
  \frac{(-1)^m}{2^{3m}}
\\
& =
  \pi^{-1/4} 2^{2l-1/2} (2l)!^{1/2} 
  \sum_{m=0}^{l}
  \frac{\pi^{1/2} (2l-2m)!}{m! (2l-2m)! 2^{2l-2m} (l-m)!}
  \cdot
  \frac{(-1)^m}{2^{3m}} 
\\
& =
  \pi^{1/4} 2^{-1/2} (2l)!^{1/2} 
  \sum_{m=0}^l
  \frac{1}{m! (l-m)!}
  \cdot
  \frac{(-1)^m}{2^m}
\\
& =
  \frac{\pi^{1/4} (2l)!^{1/2}}{2^{l+1/2} l!},
\\
  a_{2l+1}(0)
& = 
  \pi^{-1/4} 2^{2l+1/2} (2l+1)!^{1/2} 
  \sum_{m=0}^{l}
  \frac{ \Gamma(l-m+1)}{m! (2l-2m+1)!}
  \cdot
  \frac{(-1)^m}{2^{3m}}
\\
& = 
  \pi^{-1/4} 2^{2l+1/2} (2l+1)!^{1/2} 
  \sum_{m=0}^{l}
  \frac{(l-m)!}{m! (2l-2m+1)!}
  \cdot
  \frac{(-1)^m}{2^{3m}}
\end{align*}
for any nonnegative integer $l$. 
This completes the proof. 
\end{proof} 
We can also compute the Hermite expansion of 
$\lvert{x}\rvert^\lambda=x_+^\lambda+(-x)_+^\lambda$ 
for $\operatorname{Re}(\lambda)>-1$. 
\begin{corollary}
\label{theorem:abs} 
Let $\lambda$ be a complex number satisfying $\operatorname{Re}(\lambda)>-1$. 
If we set 
$$
\lvert{x}\rvert^\lambda 
= 
\sum_{k=0}^\infty 
a_k(\lambda) 
h_k 
\quad\text{in}\quad
\mathscr{S}^\prime(\mathbb{R}), 
$$
then for any nonnegative integer $l$, $a_{2l+1}(\lambda)=0$ and 
$$
a_{2l}(\lambda)
= 
\pi^{-1/4} 
2^{2l+(\lambda+1)/2} 
(2l)!^{1/2} 
\sum_{m=0}^{l} 
\cfrac{\Gamma\left(l-m+\dfrac{\lambda+1}{2}\right)}{m! (2l-2m)!}
\cdot
\frac{(-1)^m}{2^{3m}}.
$$
\end{corollary}
\begin{proof}
Combining Theorem~\ref{theorem:complex-power} 
and the fact $h_k(-x)=(-1)^kh_k(x)$, 
we can prove Corollary~\ref{theorem:abs}. 
Here we omit the detail.  
\end{proof}
We can also obtain the Hermite expansion of $\operatorname{sgn}(x)=Y(x)-Y(-x)$. 
See \cite[Proposition~2.5]{kagawa} also. 
\begin{corollary}
\label{theorem:sgn}
If we set 
$$
\operatorname{sgn}(x)
=
\sum_{k=0}^\infty 
a_k h_k
\quad\text{in}\quad
\mathscr{S}^\prime(\mathbb{R}), 
$$
then for any nonnegative integer $l$, 
$$
a_{2l}
=
0, 
\quad
a_{2l+1}
= 
\pi^{-1/4} 2^{2l+3/2} (2l+1)!^{1/2} 
\sum_{m=0}^l 
\frac{(l-m)!}{m! (2l-2m+1)!}
\cdot
\frac{(-1)^m}{2^{3m}}.
$$
\end{corollary}
\begin{proof}
Combining Theorem~\ref{theorem:complex-power} and the fact $h_k(-x)=(-1)^kh_k(x)$, 
we can prove Corollary~\ref{theorem:sgn}. Here we omit the detail.  
\end{proof}
Let $c$ be a real number. 
We consider $\delta_c$ which is said to be the Dirac measure at $c$ in $\mathbb{R}$. 
The case of $c=0$ was computed in \cite[Lemma~2.1]{kagawa}. 
\begin{theorem}
\label{theorem:diracmeasure1} 
Let $c$ be a real number. 
If we set 
$$
\delta_c 
= 
\sum_{k=0}^\infty 
a_k(c) h_k
\quad\text{in}\quad 
\mathscr{S}^\prime(\mathbb{R}), 
$$
then for any nonnegative integer $k$, 
$$
a_k(c) 
= 
\pi^{-1/4} 2^{k/2} k!^{1/2} e^{-c^2/2} 
\sum_{m=0}^{[k/2]} 
\frac{1}{m! (k-2m)!}
\cdot
\frac{(-1)^m c^{k-2m}}{2^{2m}}. 
$$
In particular, regarding $\delta_0$, we have for any nonnegative integer $l$, 
$$
a_{2l}(0) 
= 
\frac{(-1)^l (2l)!^{1/2}}{\pi^{1/4} 2^l l!}, 
\quad
a_{2l+1}(0)=0.
$$
\end{theorem}
\begin{proof}
By using the Taylor expansion of the exponential function, we have 
\begin{align*}
  B(\delta_c)(z)
& =
  2^{-1/2} \pi^{-3/4} e^{-z^2/4+cz-c^2/2}
\\
& =
  2^{-1/2} \pi^{-3/4} e^{-c^2/2}
  \sum_{\mu=0}^\infty
  \frac{(-1)^\mu z^{2\mu}}{\mu! 2^{2\mu}} 
  \sum_{\nu=0}^\infty 
  \frac{c^\nu z^\nu}{\nu!}.  
\end{align*}
If we rearrange the summations by setting $k=2\mu+\nu$ and $m=\mu$, we have 
$$
B(\delta_c)(z)
=
\sum_{k=0}^\infty
\left\{
2^{-1/2} \pi^{-3/4} e^{-c^2} 
\sum_{m=0}^{[k/2]} 
\frac{1}{m! (k-2m)!}
\cdot
\frac{(-1)^m c^{k-2m}}{2^{2m}}
\right\}
z^k.
$$
Then we obtain for any nonnegative integer $k$, 
\begin{align*}
  a_k(c)
& =
  \pi^{1/2} 2^{k/2+1/2} k!^{1/2} 
  \times
  2^{-1/2} \pi^{-3/4} e^{-c^2} 
  \sum_{m=0}^{[k/2]} 
  \frac{1}{m! (k-2m)!}
  \cdot
  \frac{(-1)^m c^{k-2m}}{2^{2m}} 
\\
& =
  \pi^{-1/4} 2^{k/2} k!^{1/2} e^{-c^2/2} 
  \sum_{m=0}^{[k/2]} 
  \frac{1}{m! (k-2m)!}
  \cdot
  \frac{(-1)^m c^{k-2m}}{2^{2m}}. 
\end{align*}
\par
For the special case $c=0$, substitute $c=0$ into the above. 
The essential contribution to $a_k(0)$ is given by the integer of the form $m=k/2$ 
in the summation on $m$. Thus we can compute $a_k(0)$. Here we omit the detail. 
\end{proof}
Here we recall the definition of the tempered distribution $\operatorname{vp}1/x$ in $\mathbb{R}$. 
For any $\phi\in\mathscr{S}(\mathbb{R})$, this is defined by 
$$
\left\langle
\operatorname{vp}\frac{1}{x}, \phi
\right\rangle
=
\lim_{\varepsilon\downarrow0} 
\int_{\lvert{x}\rvert>\varepsilon} 
\frac{\phi(x)}{x}
dx. 
$$
We compute the Hermite expansion of this. 
See \cite[Proposition~2.6]{kagawa} also. 
\begin{theorem}
\label{theorem:vp}
If we set 
$$
\operatorname{vp}\frac{1}{x}
= 
\sum_{k=0}^\infty
a_k h_k 
\quad\text{in}\quad
\mathscr{S}^\prime(\mathbb{R}), 
$$
then for any nonnegative integer $l$, 
$$
a_{2l}
=
0, 
\quad
a_{2l+1}
=
2\pi^{1/4}(2l+1)!^{1/2} 
\sum_{m=0}^l 
\frac{1}{m! (l-m)! (2l-2m+1)} 
\cdot
\frac{(-1)^m}{2^m}.
$$
\end{theorem}
\begin{proof}
The definition of the principal value implies that 
\begin{equation}
B\left(\operatorname{vp}\frac{1}{x}\right)(z)
=
2^{-1/2} \pi^{-3/4} e^{-z^2/4} 
\lim_{\varepsilon\downarrow0} 
\int_{\lvert{x}\rvert>\varepsilon}
\frac{e^{zx-x^2/2}}{x}
dx.
\label{equation:vp1} 
\end{equation}
We modify the integration before taking the limit. 
By using the change of variable $y=-x$ for $x<0$, and the Taylor expansion of 
$e^{\pm zx}$, we deduce 
\begin{align*}
  \int_{\lvert{x}\rvert>\varepsilon}
  \frac{e^{zx-x^2/2}}{x}
  dx
& =
  \int_\varepsilon^\infty
  \frac{e^{zx}-e^{-zx}}{x}
  e^{-x^2/2}
  dx
\\
& =
  \sum_{\mu=0}^\infty
  \frac{\{1-(-1)^\mu\}z^\mu}{\mu!}
  \int_\varepsilon^\infty
  x^{\mu-1}
  e^{-x^2/2}
  dx
\\
& =
  \sum_{\nu=0}^\infty
  \frac{2z^{2\nu+1}}{(2\nu+1)!}
  \int_\varepsilon^\infty
  x^{2\nu}
  e^{-x^2/2}
  dx.  
\end{align*}
Substitute this into \eqref{equation:vp1} and take the limit. 
Then we have 
\begin{align*}
  B\left(\operatorname{vp}\frac{1}{x}\right)(z)
& =
  2^{-1/2} \pi^{-3/4} e^{-z^2/4} 
  \sum_{\nu=0}^\infty
  \frac{2I(2\nu) z^{2\nu+1}}{(2\nu+1)!}
\\
& =
  2^{-1/2} \pi^{-3/4}
  \sum_{\mu=0}^\infty
  \frac{(-1)^\mu z^{2\mu}}{\mu! 2^{2\mu}} 
  \sum_{\nu=0}^\infty
  \frac{2^{1/2}\pi^1/2 (2\nu)! z^{2\nu+1}}{(2\nu+1)! \nu! 2^\nu}
\\
& = 
  \pi^{-1/4}
  \sum_{\mu=0}^\infty
  \frac{(-1)^\mu z^{2\mu}}{\mu! 2^{2\mu}} 
  \sum_{\nu=0}^\infty
  \frac{(2\nu)! z^{2\nu+1}}{(2\nu+1)! \nu! 2^\nu}
\\
& = 
  \pi^{-1/4}
  \sum_{\mu=0}^\infty
  \frac{(-1)^\mu z^{2\mu}}{\mu! 2^{2\mu}} 
  \sum_{\nu=0}^\infty
  \frac{z^{2\nu+1}}{(2\nu+1) \nu! 2^\nu}. 
\end{align*}
If we rearrange the order of summations by setting 
$l=\mu+\nu$ and $m=\mu$, we have
\begin{align*}
  B\left(\operatorname{vp}\frac{1}{x}\right)(z)
& =
  \sum_{l=0}^\infty
  \left\{
  \pi^{-1/4}
  \sum_{m=0}^l
  \frac{1}{m! (l-m)! (2l-2m+1)}
  \cdot
  \frac{(-1)^\mu}{2^{l+m}}
  \right\}
  z^{2l+1}
\\
& =
  \sum_{l=0}^\infty
  \left\{
  \pi^{-1/4}2^{-l}
  \sum_{m=0}^l
  \frac{1}{m! (l-m)! (2l-2m+1)}
  \cdot
  \left(-\frac{1}{2}\right)^m
  \right\}
  z^{2l+1}.  
\end{align*}
By using this we obtain $a_{2l}=0$ and 
\begin{align*}
  a_{2l+1}
& =
  \pi^{1/2} 2^{l+1} (2l+1)!^{1/2}
  \times 
  \pi^{-1/4}2^{-l}
  \sum_{m=0}^l
  \frac{1}{m! (l-m)! (2l-2m+1)}
  \cdot
  \left(-\frac{1}{2}\right)^m
\\
& =
  2\pi^{1/4} (2l+1)!^{1/2} 
  \frac{1}{m! (l-m)! (2l-2m+1)}
  \cdot
  \left(-\frac{1}{2}\right)^m 
\end{align*}
for any nonnegative integer $l$. 
This completes the proof.  
\end{proof} 
Tempered distributions $(x \pm i0)^{\lambda}$ 
for $\operatorname{Re}(\lambda)>-1$ or $\lambda=-1$ are defined by 
$$
\langle
(x \pm i0)^\lambda, \phi
\rangle 
= 
\lim_{\varepsilon\downarrow0}
\int_{\mathbb{R}} 
(x \pm i \varepsilon)^{\lambda} \phi(x)
dx
$$
for any $\phi\in\mathscr{S}(\mathbb{R})$. 
It is well-known that 
\begin{alignat*}{2}
  \log(x \pm i0)
& =
  \log{\lvert{x}\rvert} \pm i\pi Y(-x) 
&
& \quad\text{in}\quad
  \mathscr{S}^\prime(\mathbb{R}),
\\
  (x \pm i0)^\lambda
& =
  e^{\lambda\log(x \pm i0)}
  = 
  x_+^\lambda + e^{\pm i\lambda\pi} (-x)_+^\lambda
& 
& \quad\text{in}\quad
  \mathscr{S}^\prime(\mathbb{R})
  \quad\text{for}\quad
  \operatorname{Re}(\lambda)>-1,
\\
  \frac{1}{x \pm i0} 
& = 
  \operatorname{vp}\frac{1}{x} \mp i\pi\delta_0 
&
& \quad\text{in}\quad
  \mathscr{S}^\prime(\mathbb{R}).
\end{alignat*}
We compute the Hermite expansion of these. 
See \cite[Proposition~2.4]{kagawa} for $\lambda=-1$ also. 
\begin{corollary}
\label{theorem:xpmi0}
Let $\lambda$ be a complex number 
satisfying $\operatorname{Re}(\lambda)>-1$ or $\lambda=-1$. 
If we set 
$$
(x \pm i0)^\lambda 
= 
\sum_{k=0}^\infty 
a_k^\pm(\lambda) h_k
\quad\text{in}\quad
\mathscr{S}^\prime(\mathbb{R}), 
$$
then for any nonnegative integers $k$ and $l$, and $\operatorname{Re}(\lambda)>-1$, 
\begin{align*}
  a_k^\pm(\lambda)
& =
  \pi^{-1/4} 2^{k+(\lambda-1)/2} k!^{1/2} \bigl\{1+(-1)^k e^{\pm i \lambda\pi}\bigr\}
  \sum_{m=0}^{[k/2]} 
  \cfrac{\Gamma\left(\dfrac{k+\lambda+1}{2}-m\right)}{m! (k-2m)!}
  \cdot
  \frac{(-1)^m}{2^{3m}}, 
\\
  a_{2l}^\pm(-1) 
& = 
  \mp i \pi^{3/4} \frac{(-1)^l (2l)!^{1/2}}{2^l l!}, 
\\
  a_{2l+1}^\pm(-1)
& =
  2\pi^{1/4}(2l+1)!^{1/2} 
  \sum_{m=0}^l 
  \frac{1}{m! (l-m)! (2l-2m+1)!}
  \cdot
  \frac{(-1)^m}{2^m}.
\end{align*}
\end{corollary}
\begin{proof}
Corollary~\ref{theorem:xpmi0} immediately follows from 
Theorems~\ref{theorem:complex-power}, 
\ref{theorem:diracmeasure1} and \ref{theorem:vp}, 
and the fact $h_k(-x)=(-1)^kh_k(x)$ for $k\in\mathbb{N}\cup\{0\}$. 
Here we omit the detail. 
\end{proof}
\section{Higher dimensional cases}
\label{section:higher-dimensional-cases}
Let $n$ be an integer not smaller than two. 
In this section 
we compute Hermite expansions of some tempered distributions in $\mathbb{R}^n$. 
We begin with the Hermite expansions of the $n$-dimensional Dirac measure. 
\begin{theorem}
\label{theorem:diracmeasure2}
Let $c\in\mathbb{R}^n$. If we set 
$$
\delta_c
=
\sum_{\alpha} 
a_\alpha(c) h_\alpha, 
\quad\text{in}\quad
\mathscr{S}^\prime(\mathbb{R}^n),
$$
then for any $\alpha\in(\mathbb{N}\cup\{0\})^n$, 
$$
a_\alpha(c)
= 
\pi^{-n/4} 2^{\lvert\alpha\rvert/2} \alpha!^{1/2} e^{-c^2/2} 
\sum_{\beta\leqslant\alpha/2} 
\frac{1}{\beta! (\alpha-2\beta)!}
\cdot
\frac{(-1)^{\lvert\beta\rvert} c^{\alpha-2\beta}}{2^{2\lvert\beta\rvert}}. 
$$
In particular, 
regarding $\delta_0$, 
we have 
\begin{alignat*}{2}
  a_{2\beta}(0)
& =
  \pi^{-n/4}
  \frac{(2\beta)!^{1/2}}{\beta!}\cdot\left(-\frac{1}{2}\right)^{\lvert\beta\rvert}
& 
& \quad\text{for}\quad
  \beta\in(\mathbb{N}\cup\{0\})^n,
\\
  a_\alpha(0)
& =
  0
& 
& \quad\text{for}\quad
  \frac{\alpha}{2}\not\in(\mathbb{N}\cup\{0\})^n.
\end{alignat*}
\end{theorem}
\begin{proof}
One can prove 
Theorem~\ref{theorem:diracmeasure2} 
by taking the product of the results of Theorem~\ref{theorem:diracmeasure1}, 
or by direct computation 
which is essentially same as that of the proof of Theorem~\ref{theorem:diracmeasure1}. 
Here we omit the detail. 
\end{proof}
We consider functions related with singular integrals. 
\begin{theorem}
\label{theorem:singular1}
Let $\lambda$ be a complex number satisfying $\operatorname{Re}(\lambda)>-n$. 
If we set 
$$
\lvert{x}\rvert^\lambda
= 
\sum_{\alpha} 
a_\alpha(\lambda) h_\alpha
\quad\text{in}\quad
\mathscr{S}^\prime(\mathbb{R}^n), 
$$
then we have 
\begin{alignat*}{2}
  a_{2\beta}(\lambda)
& =
  \pi^{-n/4} 2^{2\lvert\beta\rvert+(\lambda+n)/2-1} (2\beta)!^{1/2} 
& 
& 
\\
& \times
  \sum_{\gamma\leqslant\beta}
  \cfrac{\Gamma\left(\lvert\beta-\gamma\rvert+\cfrac{\lambda+n}{2}\right)}{\gamma! (2\beta-2\gamma)!}
  \cdot
  \frac{(-1)^{\lvert\gamma\rvert}}{2^{3\lvert\gamma\rvert}}
  \cdot
  M(n,1,2\beta-2\gamma)
&
& \quad\text{for}\quad
  \beta\in(\mathbb{N}\cup\{0\})^n,
\\
  a_\alpha(\lambda)
& =
  0
& 
& \quad\text{for}\quad
  \frac{\alpha}{2}\not\in(\mathbb{N}\cup\{0\})^n. 
\end{alignat*}
\end{theorem}
\begin{proof}
By using the polar coordinates 
$(r,\omega)\in[0,\infty)\times\mathbb{S}^{n-1}$ 
for $x=r\omega\in\mathbb{R}^n$, 
and the Taylor expansion of the exponential function, 
we deduce 
\begin{align*}
&  B(\lvert{x}\rvert^\lambda)(z)
\\
  =
& 2^{-n/2} \pi^{-3n/4} e^{-z^2/4}
  \int_{\mathbb{R}^n}
  e^{zx-x^2/2} \lvert{x}\rvert^\lambda
  dx
\\
  =
& 2^{-n/2} \pi^{-3n/4} e^{-z^2/4}
  \int_0^\infty
  \left(
  \int_{\mathbb{S}^{n-1}}
  e^{rz\omega-r^2/2}
  r^{\lambda+n-1}
  d\sigma(\omega)
  \right)
  dr
\\
\\
  =
& 2^{-n/2} \pi^{-3n/4} e^{-z^2/4}
  \int_0^\infty
  \left\{
  \int_{\mathbb{S}^{n-1}}
  e^{rz\omega}
  d\sigma(\omega)
  \right\}
  r^{\lambda+n-1}e^{-r^2/2}
  dr
\\
  =
& 2^{-n/2} \pi^{-3n/4} e^{-z^2/4}
  \sum_\nu
  \frac{z^\nu}{\nu!}
  \int_{\mathbb{S}^{n-1}}
  \omega^\nu
  d\sigma(\omega)
  \int_0^\infty
  r^{\lvert\nu\rvert+\lambda+n-1}e^{-r^2/2}
  dr
\\
  =
& 2^{-n/2} \pi^{-3n/4} e^{-z^2/4}
  \sum_\nu
  \frac{M(n,1,2\nu) z^{2\nu}}{(2\nu)!}
  \int_0^\infty
  r^{2\lvert\nu\rvert+\lambda+n-1}e^{-r^2/2}
  dr
\\
  =
& 2^{-n/2} \pi^{-3n/4}
  \sum_\mu
  \frac{(-1)^{\lvert\mu\rvert} z^{2\mu}}{\mu! 2^{2\lvert\mu\rvert}}
  \sum_\nu
  \cfrac{M(n,1,2\nu) 2^{\lvert\nu\rvert+(\lambda+n)/2-1} \Gamma\left(\lvert\nu\rvert+\cfrac{\lambda+n}{2}\right) z^{2\nu}}{(2\nu)!}.
 \end{align*}
If we rearrange the order of summations by setting 
$\beta=\mu+\nu$ and $\gamma=\mu$, we have
\begin{align*}
&  B(\lvert{x}\rvert^\lambda)(z)
\\
  =
& \sum_{\beta}
  \left\{
  2^{-n/2+(\lambda+n)/2-1} \pi^{-3n/4}
  \sum_{\gamma\leqslant\beta}
  \cfrac{\Gamma\left(\lvert{\beta-\gamma}\rvert+\cfrac{\lambda+n}{2}\right)}{\gamma! (2\beta-2\gamma)!}
  \cdot
  \frac{(-1)^{\lvert\gamma\rvert} 2^{\lvert\beta-\gamma\rvert} M(n,1,2\beta-2\gamma)}{2^{2\lvert\gamma\rvert}}
  \right\}
  z^{2\beta}
\\
  =
& \sum_{\beta}
  \left\{
  2^{-n/2+\lvert\beta\rvert+(\lambda+n)/2-1} \pi^{-3n/4}
  \sum_{\gamma\leqslant\beta}
  \cfrac{\Gamma\left(\lvert{\beta-\gamma}\rvert+\cfrac{\lambda+n}{2}\right)}{\gamma! (2\beta-2\gamma)!}
  \cdot
  \frac{(-1)^{\lvert\gamma\rvert} M(n,1,2\beta-2\gamma)}{2^{3\lvert\gamma\rvert}}
  \right\}
  z^{2\beta}. 
\end{align*}
Hence we have $a_\alpha(\lambda)=0$ for $\alpha/2\not\in(\mathbb{N}\cup\{0\})^n$, 
and for any $\beta\in(\mathbb{N}\cup\{0\})^n$ 
\begin{align*}
  a_{2\beta}(\lambda)
& =
  \pi^{n/2} 2^{\lvert\beta\rvert+n/2} (2\beta)!^{1/2}
  \times 
  2^{-n/2+\lvert\beta\rvert+(\lambda+n)/2-1} \pi^{-3n/4}
\\
& \times
  \sum_{\gamma\leqslant\beta}
  \cfrac{\Gamma\left(\lvert{\beta-\gamma}\rvert+\cfrac{\lambda+n}{2}\right)}{\gamma! (2\beta-2\gamma)!}
  \cdot
  \frac{(-1)^{\lvert\gamma\rvert} M(n,1,2\beta-2\gamma)}{2^{3\lvert\gamma\rvert}}
\\
& =
  \pi^{-n/4} 2^{2\lvert\beta\rvert+(\lambda+n)/2-1} (2\beta)!^{1/2}
\\
& \times 
  \sum_{\gamma\leqslant\beta}
  \cfrac{\Gamma\left(\lvert{\beta-\gamma}\rvert+\cfrac{\lambda+n}{2}\right)}{\gamma! (2\beta-2\gamma)!}
  \cdot
  \frac{(-1)^{\lvert\gamma\rvert} M(n,1,2\beta-2\gamma)}{2^{3\lvert\gamma\rvert}}.
\end{align*}
This completes the proof.
\end{proof}
We consider the critical case of kernel functions of singular integrals. 
Suppose that $K(x)$ is a smooth function on $\mathbb{R}^n\setminus\{0\}$, 
which is homogeneous of degree zero and satisfies the vanishing condition on the sphere 
\begin{equation}
\int_{\mathbb{S}^{n-1}}
K(\omega)
d\sigma(\omega)
=
0.
\label{equation:vanishing} 
\end{equation}
This condition guarantees the existence of a limit 
$$
\left\langle
\operatorname{vp}\frac{K(x)}{\lvert{x}\rvert^n},\phi
\right\rangle
=
\lim_{\varepsilon\downarrow0}
\int_{\lvert{x}\rvert>\varepsilon}
\frac{K(x)}{\lvert{x}\rvert^n}
\phi(x)
dx
$$
for any $\phi\in\mathscr{S}(\mathbb{R}^n)$. 
We compute the Hermite expansion of 
$\operatorname{vp} K(x)/\lvert{x}\rvert^n$. 
\begin{theorem}
\label{theorem:singular2}
Suppose that $K(x)$ is smooth and homogeneous of degree zero on 
$\mathbb{R}^n\setminus\{0\}$, and satisfies {\rm \eqref{equation:vanishing}}. 
If we set 
$$
\operatorname{vp}\frac{K(x)}{\lvert{x}\rvert^n} 
= 
\sum_{\alpha} 
a_\alpha h_\alpha 
\quad\text{in}\quad
\mathscr{S}^\prime(\mathbb{R}^n), 
$$
then $a_0=0$, and for any $\alpha\ne0$, 
$$
a_\alpha
=
\pi^{-n/4} 2^{\lvert\alpha\rvert-1} \alpha!^{1/2} 
\sum_{\gamma<\alpha/2} 
\cfrac{\Gamma\left(\dfrac{\lvert\alpha-2\gamma\rvert}{2}\right)}{\gamma! (\alpha-2\gamma)!}
\cdot 
\frac{(-1)^{\lvert\gamma\rvert}}{2^{3\lvert\gamma\rvert}}
\cdot
M(n,K,\alpha-2\gamma).   
$$
\end{theorem}
\begin{proof}
The definition of the principal value implies that 
\begin{equation}
B\left(\operatorname{vp}\frac{K(x)}{\lvert{x}\rvert^n}\right)(z)
=
2^{-n/2} \pi^{-3n/4} e^{-z^2/4} 
\lim_{\varepsilon\downarrow0} 
\int_{\lvert{x}\rvert>\varepsilon}
\frac{e^{zx-x^2/2} K(x)}{\lvert{x}\rvert^n}
dx.
\label{equation:vp2} 
\end{equation}
We modify the integration before taking the limit. 
By using the polar coordinates 
$(r,\omega)\in[0,\infty)\times\mathbb{S}^{n-1}$ 
for $x=r\omega\in\mathbb{R}^n$, 
the Taylor expansion of $e^{zx}$ 
and the vanishing condition \eqref{equation:vanishing}, 
we deduce 
\begin{align*}
  \int_{\lvert{x}\rvert>\varepsilon}
  \frac{e^{zx-x^2/2} K(x)}{\lvert{x}\rvert^n}
  dx
& =
  \int_\varepsilon^\infty
  \left(
  \int_{\mathbb{S}^{n-1}}
  \frac{e^{rz\omega-r^2/2} K(\omega)}{r}
  d\sigma(\omega)
  \right)
  dr
\\
& =
  \sum_{\nu}
  \frac{z^\nu}{\nu!}
  \int_{\mathbb{S}^{n-1}}
  \omega^\nu K(\omega)
  d\sigma(\omega) 
  \int_\varepsilon^\infty
  r^{\lvert\nu\rvert-1} e^{-r^2/2}
  dr
\\
& =
  \sum_{\nu\ne0}
  \frac{M(n,K,\nu) z^\nu}{\nu!}
  \int_\varepsilon^\infty
  r^{\lvert\nu\rvert-1} e^{-r^2/2}
  dr.
\end{align*}
Substitute this into \eqref{equation:vp2} and take the limit. 
The we have 
\begin{align*}
  B\left(\operatorname{vp}\frac{K(x)}{\lvert{x}\rvert^n}\right)(z)
& =
  2^{-n/2} \pi^{-3n/4} e^{-z^2/4} 
  \sum_{\nu\ne0}
  \frac{M(n,K,\nu) I(\lvert\nu\rvert-1) z^\nu}{\nu!}
\\
& =
  2^{-n/2} \pi^{-3n/4} 
  \sum_\mu
  \frac{(-1)^{\lvert\mu\rvert} z^{2\mu}}{\mu! 2^{2\lvert\mu\rvert}}
  \sum_{\nu\ne0}
  \cfrac{M(n,K,\nu) 2^{\lvert\nu\rvert/2-1} \Gamma\left(\cfrac{\lvert\nu\rvert}{2}\right) z^\nu}{\nu!}.
\end{align*}
If we rearrange the order of summations by setting 
$\alpha=2\mu+\nu$ and $\gamma=\mu$, we have
\begin{align*}
& B\left(\operatorname{vp}\frac{K(x)}{\lvert{x}\rvert^n}\right)(z)
\\
  =
& \sum_{\alpha\ne0}
  \left\{
  2^{-n/2} \pi^{-3n/4} 
  \sum_{\gamma<\alpha/2}
  \cfrac{\Gamma\left(\cfrac{\lvert\alpha-2\gamma\rvert}{2}\right)}{\gamma! (\alpha-2\gamma)!}
  \cdot
  \frac{(-1)^{\lvert\gamma\rvert} M(n,K,\alpha-2\gamma) 2^{\lvert\alpha-2\gamma\rvert/2-1}}{2^{2\lvert\gamma\rvert}}
  \right\}
  z^\alpha 
\\
  =
& \sum_{\alpha\ne0}
  \left\{
  2^{-n/2+\lvert\alpha\rvert/2-1} \pi^{-3n/4} 
  \sum_{\gamma<\alpha/2}
  \cfrac{\Gamma\left(\cfrac{\lvert\alpha-2\gamma\rvert}{2}\right)}{\gamma! (\alpha-2\gamma)!}
  \cdot
  \frac{(-1)^{\lvert\gamma\rvert} M(n,K,\alpha-2\gamma)}{2^{3\lvert\gamma\rvert}}
  \right\}
  z^\alpha.  
\end{align*}
By using this, we obtain $a_0=0$ and 
\begin{align*}
  a_\alpha
& =
  \pi^{n/2} 2^{\lvert\alpha\rvert/2+n/2} \alpha!^{1/2} 
  \times
  2^{-n/2+\lvert\alpha\rvert/2-1} \pi^{-3n/4} 
\\
& \times 
  \sum_{\gamma<\alpha/2}
  \cfrac{\Gamma\left(\cfrac{\lvert\alpha-2\gamma\rvert}{2}\right)}{\gamma! (\alpha-2\gamma)!}
  \cdot
  \frac{(-1)^{\lvert\gamma\rvert} M(n,K,\alpha-2\gamma)}{2^{3\lvert\gamma\rvert}}
\\
& =
  \pi^{-n/4} 2^{\lvert\alpha\rvert-1} \alpha!^{1/2}
  \sum_{\gamma<\alpha/2}
  \cfrac{\Gamma\left(\cfrac{\lvert\alpha-2\gamma\rvert}{2}\right)}{\gamma! (\alpha-2\gamma)!}
  \cdot
  \frac{(-1)^{\lvert\gamma\rvert} M(n,K,\alpha-2\gamma)}{2^{3\lvert\gamma\rvert}}
\end{align*}
for any $\alpha\ne0$. 
This completes the proof.
\end{proof}
For $\rho>0$, 
set $\rho\mathbb{S}^{n-1}=\{x\in\mathbb{R}^n\ \vert \ \lvert{x}\rvert=\rho\}$. 
We denote by $d\sigma(S)$ the standard volume element of a hypersurface $S$. 
Finally we compute the Hermite expansion of 
surface carried measure on $\rho\mathbb{S}^{n-1}$. 
\begin{theorem}
\label{theorem:surface1} 
Let $\rho>0$. If we set 
$$
d\sigma(\rho\mathbb{S}^{n-1})
=
\sum_{\alpha} 
a_\alpha(\rho) h_\alpha
\quad\text{in}\quad
\mathscr{S}^\prime(\mathbb{R}), 
$$
then we have 
\begin{alignat*}{2}
  a_{2\beta}(\rho)
& =
  \pi^{-n/4} 2^{\lvert\beta\rvert} (2\beta)!^{1/2} \rho^{n-1} e^{-\rho^2/2} 
&
&
\\
& \times 
  \sum_{\gamma\leqslant\beta} 
  \frac{1}{\gamma! (2\beta-2\gamma)!}
  \cdot
  \frac{(-1)^{\lvert\gamma\rvert} \rho^{2\lvert\beta-\gamma\rvert}}{2^{2\lvert\gamma\rvert}} 
  \cdot
  M(n,1,2\beta-2\gamma)
&
& \quad\text{for}\quad
  \beta\in(\mathbb{N}\cup\{0\}),
\\
  a_\alpha(\rho)
& =
  0
&
& \quad\text{for}\quad
  \frac{\alpha}{2}\not\in(\mathbb{N}\cup\{0\}).
\end{alignat*}
\end{theorem}
\begin{proof}
By using the coordinates $\omega\in\mathbb{S}^{n-1}$ for $x=\rho\omega$ 
and the Taylor expansion of the exponential function, we deduce 
\begin{align*}
  B\Bigl(d\sigma(\rho\mathbb{S}^{n-1})\Bigr)(z)
& =
  2^{-n/2} \pi^{-3n/4} e^{-z^2/4} 
  \int_{\rho\mathbb{S}^{n-1}}
  e^{zx-x^2/2}
  d\sigma(x)
\\
& =
  2^{-n/2} \pi^{-3n/4} \rho^{n-1} e^{-\rho^2/2} e^{-z^2/4} 
  \int_{\mathbb{S}^{n-1}}
  e^{\rho z\omega}
  d\sigma(\omega)
\\
& =
  2^{-n/2} \pi^{-3n/4} \rho^{n-1} e^{-\rho^2/2} e^{-z^2/4} 
  \sum_\nu 
  \frac{\rho^{\lvert\nu\rvert} z^{\nu}}{\nu!} 
  \int_{\mathbb{S}^{n-1}}
  \omega^\nu
  d\sigma
\\
& =
  2^{-n/2} \pi^{-3n/4} \rho^{n-1} e^{-\rho^2/2} e^{-z^2/4} 
  \sum_\nu 
  \frac{\rho^{2\lvert\nu\rvert} z^{2\nu}}{(2\nu)!} 
  \int_{\mathbb{S}^{n-1}}
  \omega^{2\nu}
  d\sigma
\\
& =
  2^{-n/2} \pi^{-3n/4} \rho^{n-1} e^{-\rho^2/2} 
  \sum_\mu
  \frac{(-1)^{\lvert\mu\rvert} z^{2\mu}}{\mu! 2^{2\lvert\mu\rvert}}
  \sum_\nu 
  \frac{\rho^{2\lvert\nu\rvert} M(n,1,2\nu) z^{2\nu}}{(2\nu)!}.  
\end{align*}
If we rearrange the order of summations by setting 
$\beta=\mu+\nu$ and $\gamma=\mu$, we have
\begin{align*}
& B\Bigl(d\sigma(\rho\mathbb{S}^{n-1})\Bigr)(z)
\\
  =
& \sum_{\beta}
  \left\{
  2^{-n/2} \pi^{-3n/4} \rho^{n-1} e^{-\rho^2/2}
  \sum_{\gamma\leqslant\beta}
  \frac{1}{\gamma! (2\beta-2\gamma)!}
  \cdot
  \frac{(-1)^{\lvert\gamma\rvert} \rho^{2\lvert\beta-\gamma\rvert} M(n,1,2\beta-2\gamma)}{  2^{2\lvert\gamma\rvert}}  
  \right\}
  z^{2\beta}. 
\end{align*}
By using this, we obtain 
$a_\alpha(\rho)=0$ for $\alpha/2\not\in(\mathbb{N}\cup\{0\})^n$ and 
\begin{align*}
  a_{2\beta}(\rho)
& =
  \pi^{n/2} 2^{\lvert\beta\rvert+n/2} (2\beta)!^{1/2} 
  \times
  2^{-n/2} \pi^{-3n/4} \rho^{n-1} e^{-\rho^2/2}
\\
& \times
  \sum_{\gamma\leqslant\beta}
  \frac{1}{\gamma! (2\beta-2\gamma)!}
  \cdot
  \frac{(-1)^{\lvert\gamma\rvert} \rho^{2\lvert\beta-\gamma\rvert} M(n,1,2\beta-2\gamma)}{  2^{2\lvert\gamma\rvert}}
\\
& =
  \pi^{-n/4} 2^{\lvert\beta\rvert} (2\beta)!^{1/2} \rho^{n-1} e^{-\rho^2/2}
  \sum_{\gamma\leqslant\beta}
  \frac{1}{\gamma! (2\beta-2\gamma)!}
  \cdot
  \frac{(-1)^{\lvert\gamma\rvert} \rho^{2\lvert\beta-\gamma\rvert} M(n,1,2\beta-2\gamma)}{  2^{2\lvert\gamma\rvert}} 
\end{align*}
for any $\beta\in(\mathbb{N}\cup\{0\})^n$. 
This completes the proof. 
\end{proof}
%
%

\end{document}